\documentclass[11pt]{amsart}
\usepackage{amssymb,amsmath,amsthm,enumerate,graphicx,caption}
\usepackage{tikz-cd}
\usepackage{soul}
\theoremstyle{plain}
\newtheorem{theorem}{Theorem}[section]
\newtheorem{lemma}[theorem]{Lemma}
\newtheorem{proposition}[theorem]{Proposition}
\newtheorem{corollary}[theorem]{Corollary}

\usepackage{soul} 
\usepackage{subcaption} 

\theoremstyle{remark}
\newtheorem*{remark}{Remark}
\newtheorem*{remarks}{Remarks}

\let\Re\undefined
\let\Im\undefined
\DeclareMathOperator{\Re}{\mathrm{Re}}
\DeclareMathOperator{\Im}{\mathrm{Im}}
\DeclareMathOperator{\tr}{\mathrm{tr}}

\begin{document}

\title[Schwarz--Jack lemma]{A Schwarz--Jack lemma, circularly symmetric domains  and  numerical ranges}

\author[J. Mashreghi]{Javad Mashreghi}
\address{D\'epartement de math\'ematiques et de statistique, Universit\'e Laval, Qu\'ebec (QC), G1V 0A6, Canada}
\email{javad.mashreghi@mat.ulaval.ca}

\author[A. Moucha]{Annika Moucha}
\address{Department of Mathematics, University of W\"urzburg, 97074 W\"urzburg, Germany}
\email{annika.moucha@uni-wuerzburg.de}

\author[R. O'Loughlin]{Ryan O'Loughlin}
\address{Department of Mathematics and Statistics, University of Reading, White\-knights, Reading RG6 6AX, UK}
\email{r.d.oloughlin@reading.ac.uk}

\author[T. Ransford]{Thomas Ransford}
\address{D\'epartement de math\'ematiques et de statistique, Universit\'e Laval, Qu\'ebec (QC), G1V 0A6, Canada}
\email{thomas.ransford@mat.ulaval.ca}

\author[O. Roth]{Oliver Roth}
\address{Department of Mathematics, University of W\"urzburg, 97074 W\"urzburg, Germany}
\email{oliver.roth@uni-wuerzburg.de}

\begin{abstract}
We prove a Schwarz--Jack lemma for holomorphic functions on the unit disk with the property that their maximum modulus on each circle about the origin is attained at a point on the positive real axis. With the help of this result, we  establish monotonicity and convexity properties  of conformal maps of circularly symmetric and
bi-circularly symmetric domains. As an application, we give a new proof of  Crouzeix's theorem that the numerical range of any  $2\times 2$ matrix is a $2$-spectral set for the matrix. Unlike other proofs, our approach does not depend on the explicit formula for the conformal mapping of an ellipse onto the unit disk.
\end{abstract}

\thanks{Mashreghi supported by an NSERC Discovery Grant and by the Canada Research Chairs program. Moucha  supported by the Centre de Recherches Math\'ematiques (Montr\'eal), the Simons Foundation and  by the Alexander von Humboldt Stiftung. O'Loughlin  supported by an EPSRC Postdoctoral Fellowship. Ransford supported  by an NSERC Discovery Grant. Roth supported by the Centre de Recherches Math\'ematiques (Montr\'eal) and the Simons Foundation.}



\maketitle

\section{Introduction}\label{S:Intro}

The celebrated Crouzeix conjecture asserts that the numerical range of any $n \times n$ matrix is a $2$-spectral
set for the matrix. While this result has been established for the case $n = 2$ (where the numerical range is always an ellipse or a line segment), the conjecture remains open for $n \geq 3$.
Even for $n=2$,  the existing proofs are rather technical, relying heavily on the explicit formula of the conformal mapping
from an ellipse onto the unit disk
(see \cite{BCD06,CGL17,Cr04}). 
The aim of this work is twofold: first, to elucidate the intrinsic properties of
the conformal mapping in question, thereby reducing the reliance on explicit formulas; second, to lay
conceptual groundwork for broader generalizations. 

To this end, we establish two auxiliary results of independent interest.
The first is a convexity theorem for holomorphic functions, 
which we have labelled a Schwarz--Jack lemma, since it
combines ideas from the classical lemmas of Schwarz and Jack.
The second is a characterization of those simply connected domains
that are circularly symmetric in terms of their Riemann mapping functions.
Using these results, we establish a convexity theorem for conformal mappings of bi-circularly symmetric domains, a large class of domains that includes ellipses. This result is then applied to 
furnish a new proof of Crouzeix's theorem for $2\times2$ matrices
that does not depend on the explicit formula for the Riemann mapping function of an ellipse.


\section{A Schwarz--Jack lemma}\label{S:SchwarzJack}

The title of this section refers to the following theorem.

\begin{theorem}\label{T:SchwarzJack}
Let $f$ be a holomorphic function on the open unit disk $\mathbb{D}$.
Suppose that, for some $n\ge1$, we have
$f(0)=f'(0)=\cdots=f^{(n-1)}(0)=0$ and $f^{(n)}(0)>0$.
Suppose further that
\begin{equation}\label{E:Jack}
|f(z)|\le \bigl|f(|z|)\bigr| \quad(z\in\mathbb{D}).
\end{equation}
Then  $f|_{[0,1)}$  is a positive, strictly increasing,
convex function.
Moreover, $f$ is strictly convex on $[0,1)$ unless $f(z)\equiv \alpha z, \alpha>0$.
\end{theorem}

We call this a Schwarz--Jack lemma since it is obtained by combining
ideas from the classical lemmas of Schwarz and Jack. (For
background on Jack's lemma, see \cite{Fo19,Ja71}.) What makes this result maybe a little unusual is the fact that part of the conclusion is of second-order, namely that $f''\ge0$ on $[0,1)$.

\begin{proof}[Proof of Theorem~\ref{T:SchwarzJack}]
We begin by showing that $f$ is strictly positive on $(0,1)$.
This part of the argument is essentially the same as in Jack's lemma.
Certainly $f\not\equiv0$ since $f^{(n)}(0)>0$,
and therefore condition~\eqref{E:Jack} implies that $f$
has no zeros in $(0,1)$.
For each $re^{i\theta}\in\mathbb{D}$ such that $f(re^{i\theta})\ne0$, we have
\begin{equation}\label{E:derivative}
\frac{\partial}{\partial\theta}\log |f(re^{i\theta})|
=\Re\Bigl(\frac{\partial}{\partial\theta}\log f(re^{i\theta})\Bigr)
=-\Im\Bigl(re^{i\theta}\frac{f'(re^{i\theta})}{f(re^{i\theta})}\Bigr).
\end{equation}
Condition~\eqref{E:Jack} implies that the left-hand side of
\eqref{E:derivative} is zero whenever $\theta=0$.
It follows from \eqref{E:derivative} that  $f'(r)/f(r)$ is real-valued for all $r\in(0,1)$.
This implies that $f$ maps $(0,1)$ into a half-line passing through
the origin. As $f(0)=f'(0)=\cdots=f^{(n-1)}(0)=0$ and $f^{(n)}(0)>0$, this half-line must be the positive real axis.

Next we show that $f$ is  increasing on $(0,1)$.
Let $0<s<r<1$.
The condition \eqref{E:Jack}
implies that $f$ maps the open disk $D(0,r)$ 
with center $0$ and radius~$r$ into the disk $D(0,f(r))$.
Since also $f(0)=0$,
 we may apply Schwarz's lemma
to  the function $z\mapsto f(rz)/f(r)$ to get
\[
|f(rz)/f(r)|\le|z| \quad(z\in\mathbb{D}).
\]
Setting $z=s/r$, we deduce that
\begin{equation}\label{E:monotone}
f(s)/f(r)\le s/r \quad(0<s<r<1).
\end{equation}
In particular, $f$ is strictly increasing on $[0,1)$.

Now we prove that $f$ is convex on $[0,1)$.
Once again, let $0<s<r<1$. This time we apply the
Schwarz--Pick lemma to $f(rz)/f(r)$, to obtain
\[
\frac{|rf'(rz)/f(r)|}{1-|f(rz)/f(r)|^2}\le \frac{1}{1-|z|^2}
\quad(z\in\mathbb{D}).
\]
Setting $z=s/r$  and rearranging, we obtain
\[
f'(s)\le \Bigl(\frac{f(r)-f(s)}{r-s}\Bigr)\Bigl(\frac{1+f(s)/f(r)}{1+s/r}\Bigr).
\]
By \eqref{E:monotone}, the second factor in  the
right-hand side is at most $1$. Hence
\[
f'(s)\le \frac{f(r)-f(s)}{r-s} \quad(0<s<r<1).
\]
Writing $r=s+h$, we see that this is equivalent to
\[
f(s+h)-f(s)-hf'(s)\ge0 \quad(0<s<s+h<1).
\]
Since
\[
f''(s)=\lim_{h\to0^+} \frac{f(s+h)-f(s)-hf'(s)}{h^2/2},
\]
we deduce that $f''(s)\ge0$ for all $s\in(0,1)$.
Thus $f$ is convex on $[0,1)$. 

Finally, if $f|_{[0,1)}$ is not strictly convex, then $f''$ vanishes on some non-empty open subinterval of $[0,1)$, 
so, by the identity principle for holomorphic functions, $f(z)\equiv \alpha z$ on $\mathbb{D}$, where $\alpha>0$.
\end{proof}

Theorem~\ref{T:SchwarzJack} can be used in various ways
to obtain information about conformal mappings.
In order to implement it, one needs to be able to verify
the condition~\eqref{E:Jack}.  Since this condition arises
in Jack's lemma, it seems natural to call it the \emph{Jack condition}.
In the next section we shall establish a simple geometric criterion
for simply connected domains guaranteeing that their Riemann maps satisfy the Jack condition.


\section{Circularly symmetric domains}\label{S:circsym}

Let $\Omega$ be a plane domain. We say that $\Omega$
is \emph{circularly symmetric} (about~$0$) if its intersection with each circle
around the origin is either empty, or the whole circle, or consists of a subarc of the circle centred on the positive $x$-axis.

Circularly symmetric domains are exactly those obtained from general ones by
applying the process of circular symmetrization with respect to 
the positive real axis.
They were studied by Jenkins in \cite{Je55}.
(His definition of circular symmetry includes some extra technical conditions, but these are 
automatically satisfied if the domain is simply connected, which
is the only case that we shall consider here.)

Our main result in this section is a characterization, via their Riemann mapping functions, of those
simply connected domains that are circularly symmetric.

\begin{theorem}\label{T:circsym}
Let $\Omega$ be a simply connected, proper subdomain of $\mathbb{C}$ containing $0$.
Let $f$ be the unique conformal mapping of $\mathbb{D}$ onto $\Omega$ such that $f(0)=0$ and $f'(0)>0$.
Then the following statements are equivalent:
\begin{enumerate}[\normalfont(i)]
\item $\Omega$ is circularly symmetric;
\item $f(\overline{z})=\overline{f(z)}$ for all $z\in\mathbb{D}$ and, for each $r\in(0,1)$, the map $\theta\mapsto|f(re^{i\theta})|$ is decreasing on $[0,\pi]$.
\end{enumerate}
\end{theorem}

As an immediate consequence of the implication (i)$\Rightarrow$(ii), we obtain the geometric criterion
for the Jack condition that we were seeking.

\begin{corollary}\label{C:circsym}
Let $\Omega$ and $f$ be as in Theorem~\ref{T:circsym}.
If $\Omega$ is circularly symmetric, then $f$ satisfies the Jack condition.
\end{corollary}

For the proof of Theorem~\ref{T:circsym}, 
we need two lemmas. The first of these is a result of Jenkins
\cite[Theorem~2]{Je55}. 

\begin{lemma}\label{L:circsym}
Let $\Omega$ and $f$ be as in Theorem~\ref{T:circsym}.
If $\Omega$ is circularly symmetric, then $f$ maps
each disk $D(0,r),0<r<1$, onto a circularly symmetric domain.
\end{lemma}

The second lemma is also implicit in \cite{Je55}, but this time we
spell out the details.

\begin{lemma}\label{L:strict}
Let $\Omega$ and $f$ be as in Theorem~\ref{T:circsym}.
If $f$ satisfies the condition~(ii) of that theorem, 
then either $f(z)\equiv \alpha z$ with $\alpha>0$ or,
for each $r\in(0,1)$, 
the map $\theta\mapsto|f(re^{i\theta})|$ is strictly decreasing on $[0,\pi]$.
\end{lemma}

\begin{proof}
As already seen in \eqref{E:derivative},  we have
\begin{equation}\label{E:derivative2}
\frac{\partial}{\partial\theta}\log |f(re^{i\theta})|
=-\Im\Bigl(re^{i\theta}\frac{f'(re^{i\theta})}{f(re^{i\theta})}\Bigr)
\quad(r\in(0,1),\theta\in[-\pi,\pi]).
\end{equation}
The condition (ii) in Theorem~\ref{T:circsym} therefore implies that
$\Im(zf'(z)/f(z))\ge0$ for all $z\in\mathbb{D}^+:=\mathbb{D}\cap\{\Im z>0\}$.
As $\Im(zf'(z)/f(z))$ is a harmonic function, if it attains a minimum 
on $\mathbb{D}^+$ then it must be constant on $\mathbb{D}^+$.
Thus either $f(z)\equiv \alpha z$ for some $\alpha>0$ or $\Im(zf'(z)/f(z))>0$ on $\mathbb{D}^+$. By \eqref{E:derivative2}, the latter condition implies that 
$\theta\mapsto|f(re^{i\theta})|$ is strictly decreasing on $[0,\pi]$
for each $r\in(0,1)$.
\end{proof}

\begin{proof}[Proof of Theorem~\ref{T:circsym}]
(i)$\Rightarrow$(ii):
If $\Omega$ is circularly symmetric, then it is certainly symmetric about the $x$-axis.
By uniqueness of $f$, we have $f(\overline{z})=\overline{f(z)}$ for all $z\in\mathbb{D}$.

Let $r\in(0,1)$. 
By Lemma~\ref{L:circsym},
$f$ maps the disk $D(0,r)$ onto a circularly symmetric domain. 
This  implies that
$\theta\mapsto|f(re^{i\theta})|$ is decreasing on $[0,\pi]$.

\medskip

(ii)$\Rightarrow$(i):
We begin with a preliminary remark. By Lemma~\ref{L:strict},
unless $f(z)\equiv f'(0)z$ (in which case the result is obvious), 
the map $\theta\mapsto|f(re^{i\theta})|$ is  strictly decreasing on $[0,\pi]$ for each $r\in(0,1)$, and likewise  strictly increasing on $[-\pi,0]$. 

Let $w\in\Omega\setminus[0,\infty)$, and let $C$ be the circle with centre $0$ and passing through~$w$.
Set $r:=|f^{-1}(w)|$. Then $f$ maps the circle $|z|=r$ onto a Jordan curve $J$ passing through $w$. 
The preliminary remark above shows that
$J$ meets $C$ precisely in the set $\{w,\overline{w}\}$.
Thus each
subarc of $C\setminus\{w,\overline{w}\}$  lies inside a component of
$\mathbb{C}\setminus J$. Furthermore, $|w|$ itself lies in the interior of~$J$.
Indeed, $|w|\in f([0,r))$, and $f([0,r))$ is a connected set containing $0$ and disjoint from $J$,
so is entirely contained in the interior of $J$.
This implies that the whole subarc of $C\setminus\{w,\overline{w}\}$ that contains $|w|$ 
lies in the interior of~$J$, and therefore in $\Omega$.
\end{proof}


\section{Bi-circularly symmetric domains}\label{S:bi-circsym}

If we combine Theorem~\ref{T:SchwarzJack} and Corollary~\ref{C:circsym}, then we immediately obtain the 
following result.

\begin{theorem}\label{T:circsymSchwarzJack}
Let $\Omega$ be a circularly symmetric, simply connected, proper
subdomain of $\mathbb{C}$ containing $0$. Let $f$
be the unique conformal mapping of $\mathbb{D}$ onto $\Omega$
satisfying $f(0)=0$ and $f'(0)>0$.
Then  $f|_{[0,1)}$  is a positive, strictly increasing,
convex function.
Moreover, $f$ is strictly convex on $[0,1)$ unless $\Omega$
is a disk with centre $0$.
\end{theorem}

Jenkins  had previously obtained a weaker version
of this result in \cite[Theorem~3]{Je55} with the conclusion $f''(0)>0$ (unless $\Omega$
is a disk with centre $0$).
Interesting though Theorem~\ref{T:circsymSchwarzJack} may be, it is not quite what we need, 
since it does not include the possibility that $\Omega$ be an ellipse.
To cover that case, we need to adjust the notion of circularly symmetric domains.
Accordingly, we make the following definition.

Let us say that a plane domain $\Omega$ is \emph{bi-circularly symmetric} if its intersection with each circle around $0$ is either empty, or the whole circle, or the union of two subarcs of equal length centred respectively on the positive and negative real axes. 

An example of such a domain is the ellipse $\{x^2/a^2+y^2/b^2<1\}$,
where $a>b>0$. There are many other examples. 
One way to generate such examples is to start with a continuous, decreasing function $R:[0,\pi/2]\to(0,\infty)$, and then let $\Omega$ be the domain that is symmetric about the $x$- and $y$-axes 
and whose intersection with the 
first quadrant $Q:=\{x\ge 0,y\ge 0\}$ satisfies
\[
\Omega\cap Q=\{re^{i\theta}: 0\le r<R(\theta),\, 0\le \theta\le\pi/2\}.
\]
Figure~\ref{F:bi-circsym}(b) and (c) shows two such examples.
Note that they are not necessarily convex and that they may have cusps. 
Figure~\ref{F:bi-circsym}(d) exhibits an  example
of a simply connected,  bi-circularly symmetric domain that
is not even star-shaped.

\begin{center}
\begin{figure}[htb]
\begin{minipage}{0.49\textwidth}
  \begin{center}
  \includegraphics[scale=0.25]{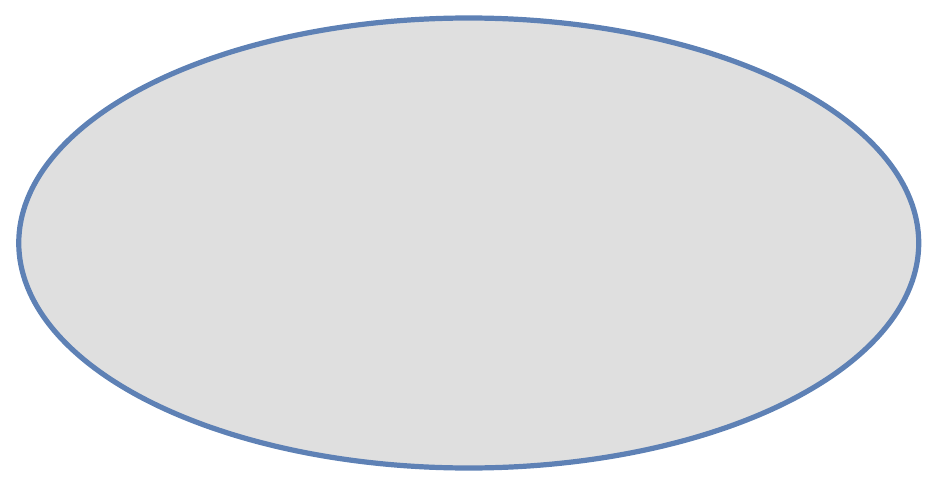}
  \par (a)
  \end{center}
\end{minipage}
\begin{minipage}{0.49\textwidth}
\begin{center}
   \includegraphics[scale=0.2]{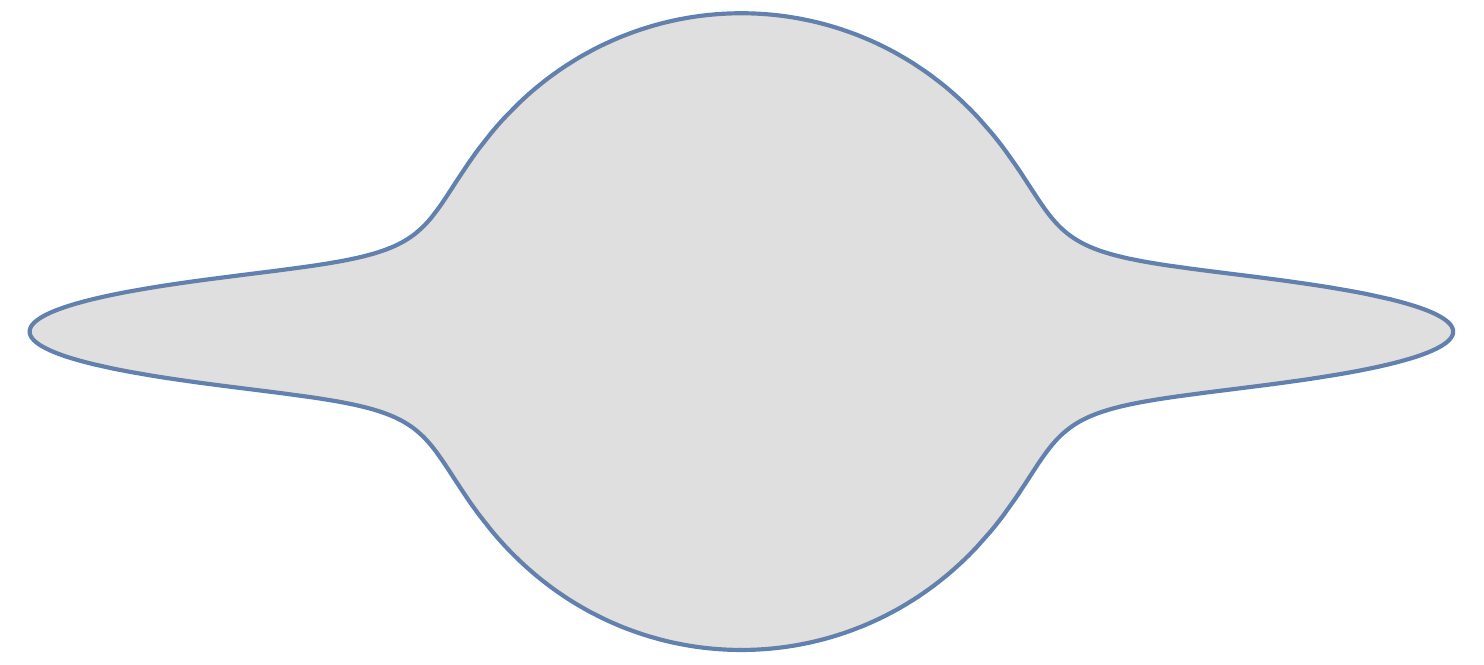}
   \par (b)
   \end{center}
\end{minipage}

\begin{minipage}{0.49\textwidth}
\begin{center}
   \includegraphics[scale=0.2]{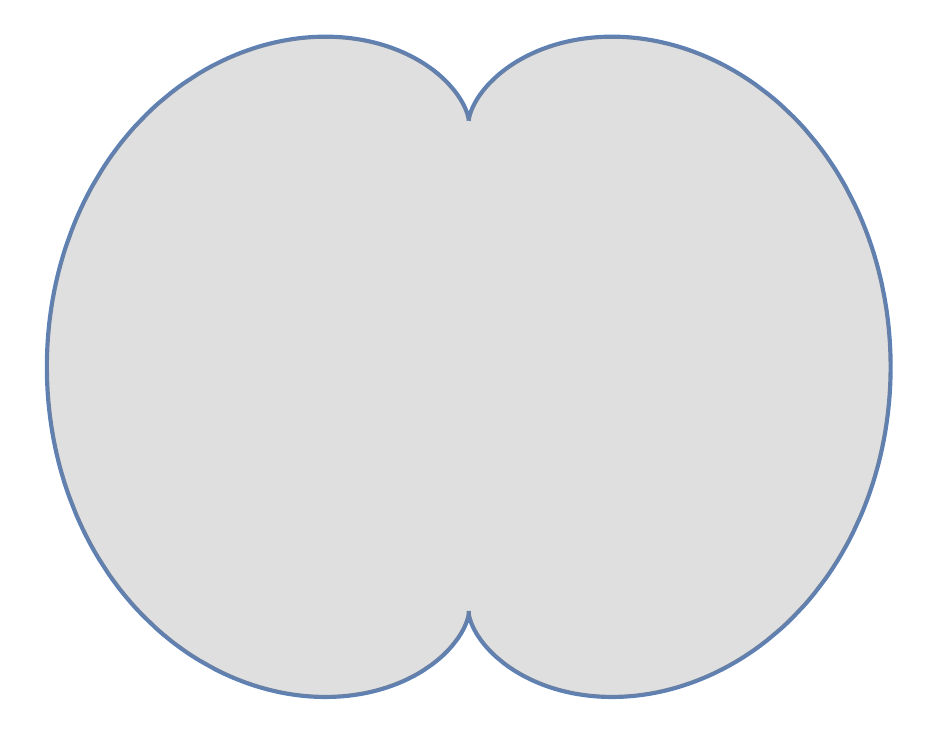}
   \par (c)
   \end{center}
\end{minipage}
\begin{minipage}{0.49\textwidth}
\begin{center}
  \includegraphics[scale=0.17]{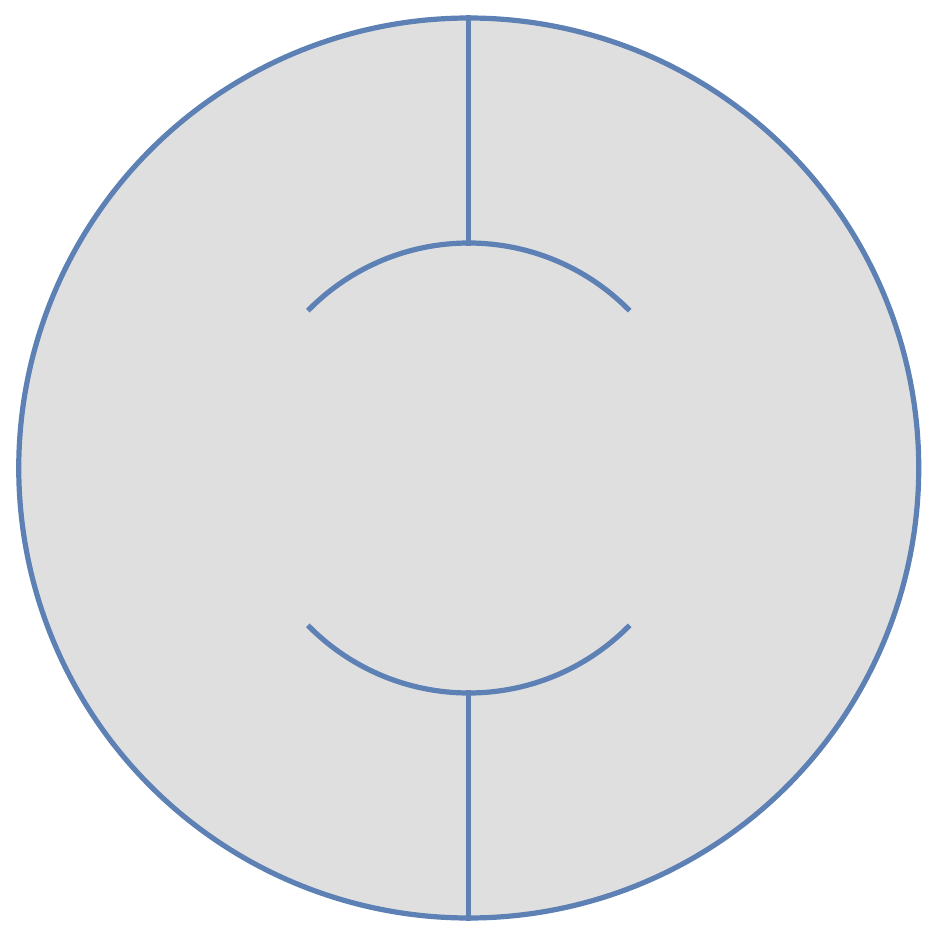}
  \par (d)
  \end{center}
\end{minipage}

\caption{Examples of bi-circularly symmetric domains:
\newline
(a) Ellipse $x^2+4y^2<1$ 
\newline
(b) Example with $R(\theta)=\sqrt{1/4+1/(1+50\theta^2)^2}$
for $\theta\in[0,\frac{\pi}{2}]$
\newline
(c) Image of $\mathbb{D}$ under $z+z^3/4-z^5/20$ (see Proposition~\ref{P:counterex})
\newline
(d) A non-star-shaped example}
\label{F:bi-circsym}
\end{figure}
\end{center}

Here is the result that we need.

\begin{theorem}\label{T:bi-circsym}
Let $\Omega$ be a bi-circularly symmetric, simply connected, proper
subdomain of $\mathbb{C}$ containing $0$. Let $f$
be the unique conformal mapping of $\mathbb{D}$ onto $\Omega$
satisfying $f(0)=0$ and $f'(0)>0$. Then:
\begin{enumerate}[\normalfont(i)]
\item $x\mapsto f(x)$ is positive, increasing and convex on $[0,1)$;
\item $y\mapsto -if(iy)$ is positive, increasing and concave on $[0,1)$.
\end{enumerate}
\end{theorem}

In establishing this result, the following
elementary lemma will be helpful.

\begin{lemma}\label{L:concaveinverse}
Let $h$ be a continuous, strictly increasing, convex function
defined on an interval $I$.
Then $h^{-1}$ is a continuous, strictly increasing, concave function on $h(I)$.
\end{lemma}

\begin{proof}
It is standard that $h^{-1}$ is continuous and strictly increasing.
It remains to prove that it is concave.
As $h$ is a convex function, its epigraph is a convex set.
The hypograph of $h^{-1}$ is the image of the epigraph of $h$
under reflection in the line $x=y$, so it too is a convex set.
It follows that $h^{-1}$ is a concave function.
\end{proof}

\begin{proof}[Proof of Theorem~\ref{T:bi-circsym}]
By assumption,  $f$ is the unique conformal
map of $\mathbb{D}$ onto $\Omega$ such that $f(0)=0$
and $f'(0)>0$. The maps $z\mapsto -f(-z)$
and $z\mapsto\overline{f(\overline{z})}$ also have these
properties, so, by uniqueness, 
they are both equal to $f$. In other words,
\begin{equation}\label{E:holosym}
f(-z)=-f(z)
\quad\text{and}\quad
f(\overline{z})=\overline{f(z)}
\qquad(z\in\mathbb{D}).
\end{equation}
In particular $f$ maps $(-1,1)$ into the real axis
and $i(-1,1)$ into the imaginary axis.
Moreover, since $f(0)=0$ and $f'(0)>0$, we must have
$f(\mathbb{D}\cap\mathbb{R}^+)=\Omega\cap\mathbb{R}^+$ and
$f(\mathbb{D}\cap i\mathbb{R}^+)=\Omega\cap i\mathbb{R}^+$. 

As $f$ is an odd function, we can write $f(z)^2=g(z^2)$,
where $g$ is the conformal mapping of $\mathbb{D}$
onto  $\Omega^2:=\{w^2:w\in\Omega\}$ with $g(0)=0$ and $g'(0)>0$
(see e.g.\ \cite[p.28]{Du83}). Since $\Omega$ is bi-circularly
symmetric, it follows that $\Omega^2$ is circularly symmetric.
Applying Theorem~\ref{T:circsym} to the pair $\Omega^2,g$,
we deduce that, for each $r>0$, 
the map $\theta\mapsto|g(re^{i\theta})|$ is decreasing on $[0,\pi]$.
In particular,
\[
\bigl|g(-|z|)\bigr|\le |g(z)|\le\bigl|g(|z|)\bigr|\quad(z\in\mathbb{D}).
\]
It follows that
\begin{equation}\label{E:doubleJack}
\bigl|f(i|z|)\bigr|\le |f(z)|\le \bigl|f(|z|)\bigr|\quad(z\in\mathbb{D}).
\end{equation}
In particular, $f$ satisfies the Jack condition \eqref{E:Jack}.

By the Schwarz--Jack lemma, Theorem~\ref{T:SchwarzJack},
$f$ is positive, strictly increasing and convex on $[0,1)$.
This proves part~(i) of the theorem.

To prove part~(ii), let us fix $a\in(0,1)$ and write $f(ai)=bi$,
where $b>0$.
We  re-apply Theorem~\ref{T:SchwarzJack},
but this time to the function $\psi(z):=-if^{-1}(ibz)$.
The bi-circular symmetry of $\Omega$ implies that the
disk of centre $0$ and radius $b$ is contained in $\Omega$,
so $\psi$ is a function defined on the unit disk.
Note also that, since $f(0)=0$ and $f'(0)>0$,
we have $\psi(0)=0$ and $\psi'(0)=b/f'(0)>0$.
Moreover $\psi$ satisfies the Jack condition~\eqref{E:Jack}.
Indeed, using the left-hand side of \eqref{E:doubleJack}, we have
\[
\bigl|f(i|\psi(z)|)\bigr|\le |f(i\psi(z))|=|ibz|=|ib|z||=|f(i\psi(|z|))|,
\]
which, together with the fact that
$y\mapsto|f(iy)|$ is increasing on $[0,1)$,
implies that $|\psi(z)|\le|\psi(|z|)|$.
Thus Theorem~\ref{T:SchwarzJack} applies,
and we deduce that $y\mapsto -if^{-1}(iby)$
is a positive, strictly increasing and convex function on $[0,1)$,
in other words, 
$y\mapsto -if^{-1}(iy)$
is  positive, strictly increasing and convex on $[0,b)$.
By Lemma~\ref{L:concaveinverse},
it follows that 
$y\mapsto -if(iy)$
is  positive, strictly increasing and concave on $[0,a)$.
Finally, as this holds for each $a\in(0,1)$, we deduce that
(ii) holds.
\end{proof}

\begin{remark}
If $f$ is a holomorphic function on $\mathbb{D}$ all of whose Taylor coefficients
are non-negative, then clearly $f$ satisfies the Jack condition  \eqref{E:Jack}
and it is positive, increasing and convex on $[0,1)$.
Non-negativity of the Taylor coefficients is a more stringent requirement
than the conditions developed in this article,
but there  is one interesting case where this does indeed happen.

Kanas--Sugawa showed in \cite{KS06} that,
if $\Omega$ is an ellipse $\{x^2/a^2+y^2/b^2<1\}$,
where $a>b>0$, and if $f$ is the conformal mapping of $\mathbb{D}$ onto $\Omega$
such that $f(0)=0$ and $f'(0)>0$,
then all the Taylor coefficients of $f$ are non-negative.
Their proof depends heavily on the explicit form of $f$,
expressed in terms of elliptic functions,
and their result does not extend to conformal maps of the unit disk
onto more general bi-circularly symmetric domains. The following proposition
provides some simple counterexamples
and also reveals the broader scope of Theorem~\ref{T:bi-circsym}.
\end{remark}

\begin{proposition}\label{P:counterex}
Let $f(z):=z+az^3-bz^5$, where $a,b$ are positive constants satisfying $3a+5b\le1$ and $ab+4b\le a$.
Then $f$ is a conformal mapping of~$\mathbb{D}$ onto a bi-circularly symmetric domain $\Omega$.
\end{proposition}

For example, we can take $f(z):=z+z^3/4-z^5/20$.
The corresponding  domain $\Omega$ is pictured in
Figure~\ref{F:bi-circsym}(c).

\begin{proof}[Proof of Proposition~\ref{P:counterex}]
If $|z|=1$, then
\[
\Bigl|\frac{zf'(z)}{f(z)}-1\Bigr|=\Bigl|\frac{2az^3-4bz^5}{z+az^3-bz^5}\Bigr|\le \frac{2a+4b}{1-(a+b)}\le1.
\]
By the maximum principle, the same inequality persists for all $z\in\mathbb{D}$,
from which it follows that $\Re (zf'(z)/f(z))\ge0$.
Together with the facts that $f(0)=0$ and $f'(0)=1$,
this implies that $f$ is univalent and that $\Omega:=f(\mathbb{D})$ is star-shaped about $0$
(see, e.g., \cite[Theorem~2.10]{Du83}).
Clearly we  have $f(-z)=-f(-z)$ and $f(\overline{z})=\overline{f(z)}$
for all $z\in\mathbb{D}$, so $\Omega$ is symmetric with respect to the $x$-axis and the $y$-axis.
A simple calculation gives
\[
\frac{d}{d\theta}|f(e^{i\theta})|^2=-4\sin(2\theta)\Bigl(a-ab-4b\cos(2\theta)\Bigr)\le0 \quad(\theta\in[0,\pi/2]),
\]
where the last inequality uses the condition that $ab+4b\le a$.
Thus
$\theta\mapsto|f(e^{i\theta})|$
is decreasing on $[0,\pi/2]$.
It follows that $\Omega$ is bi-circularly symmetric.
This completes the proof.
\end{proof}


\section{Application to numerical ranges}\label{S:nr}

Let $n\ge2$, and let  $\|\cdot\|$ and  $\langle\cdot,\cdot\rangle$  
be the usual Euclidean norm and inner product on $\mathbb{C}^n$, respectively.
Given an $n\times n$ matrix $A$, we denote by $W(A)$ its numerical range, namely
\[
W(A):=\Bigl\{\langle Ax,x\rangle : x\in\mathbb{C}^n,\,\|x\|=1\Bigr\}.
\]
It is a convex compact set containing the eigenvalues of $A$.
Crouzeix established the following result in \cite[Theorem~1.1]{Cr04}.

\begin{theorem}\label{T:Cr2x2}
If $A$ is a $2\times 2$ matrix 
and  $p$ is a polynomial, then
the operator norm of $p(A)$ satisfies
\begin{equation}\label{E:Cr2x2}
\|p(A)\|\le 2\max_{z\in W(A)}|p(z)|.
\end{equation}
\end{theorem}

In the same paper, he conjectured that the result also holds for general $n\times n$ matrices.
More than twenty years later, this is still an open problem, even for $3\times 3$ matrices.
A celebrated result of Crouzeix--Palencia \cite{CP17} shows that \eqref{E:Cr2x2} holds for all
$n\times n$ matrices if one replaces $2$ by $1+\sqrt{2}$.

In this section,  we present a proof of Theorem~\ref{T:Cr2x2}
based  on ideas of Caldwell, Greenbaum and Li taken from \cite[\S6.3]{CGL17}.
Certain elements of the proof also appear in \cite{CGH14}.
(The reference \cite{CGL17} is  a preprint on the arXiv. The  version of this paper that was
published in a journal
no longer includes the proof of Theorem~\ref{T:Cr2x2}.)
Our  proof differs from other published proofs (including that in \cite{CGL17})
in that it does not rely on an explicit formula for the conformal map of an ellipse onto a disk.

\begin{proof}[Proof of Theorem~\ref{T:Cr2x2}]
Let $A$ be a $2\times 2$ complex matrix.
It is clear that the problem of establishing \eqref{E:Cr2x2} is invariant
under the transformations $A\mapsto \alpha A+\beta I$, where $\alpha,\beta\in\mathbb{C}$,
and $A\mapsto U^*AU$, where $U$ is unitary.
Therefore, one quickly sees that it suffices to consider the case where
\begin{equation}\label{E:A}
A=
\begin{pmatrix}
1 &2b\\
0 &-1
\end{pmatrix},
\end{equation}
where $b>0$.
In this case, $W(A)$ is the ellipse $x^2/a^2+y^2/b^2\le1$,
where $a^2=b^2+1$. The foci of the ellipse are the eigenvalues $\pm 1$.

Let $\phi$ be an \emph{extremal function} for $A$, i.e.,   a function maximizing $\|\phi(A)\|$
among all continuous functions on $W(A)$ of sup-norm $1$ that are holomorphic on the interior of $W(A)$.
By \cite[Theorem~2.1]{Cr04}, $\phi$ is 
the composition of a conformal mapping of 
 $W(A)^\circ$ onto the unit disk $\mathbb{D}$ and a Blaschke product of
degree~$1$, so is itself conformal. It  can be extended to a homeomorphism of $W(A)$ onto $\overline{\mathbb{D}}$ by means of Carath\'eodory's theorem.
Our eventual aim is to show that $\|\phi(A)\|\le2$.

We next show that $\tr(\phi(A))=0$.
Set $B:=\phi(A)$ and 
let $x$ be a unit vector in $\mathbb{C}^2$ on which $B$ attains its norm. By
\cite[Proposition~2.2]{CGH14}, if $\|B\|>1$ (which we may as well suppose to be the case),
 then $\langle Bx,x\rangle=0$. Since $B^*Bx=\|B\|^2 x$, we also have $\langle B(Bx),Bx\rangle=\langle Bx,B^*Bx\rangle=\|B\|^2\langle Bx,x\rangle=0$.
 Thus the matrix of $B$ with respect to the orthogonal basis
 $\{x,Bx\}$ of $\mathbb{C}^2$ has zero entries on its diagonal, and hence $\tr(B)=0$, as claimed.

We next show that $\phi$ is an odd function, and in particular that $\phi(0)=0$.
Indeed, the eigenvalues of $\phi(A)$ are located 
at $\phi(\pm 1)$, and, as $\tr(\phi(A))=0$,
they must sum to zero. Thus $\phi(-1)=-\phi(1)$.  Consider  $m(z):=\phi(-\phi^{-1}(-z))$.
This is a  M\"obius automorphism of the unit disk, and it has
at least two fixed points, namely $\pm\phi(1)$, so it must be the identity.
Hence $\phi(-z)=-\phi(z)$, as claimed.
Multiplying $\phi$ by an appropriate unimodular constant, we may further suppose that $\phi'(0)>0$.
To summarize, $\phi$ is the unique homeomorphism 
of $W(A)$ onto $\overline{\mathbb{D}}$ mapping $W(A)^\circ$ 
conformally onto $\mathbb{D}$ and such that
$\phi(0)=0$ and $\phi'(0)>0$.

The next step is to invoke \cite[Lemma~2.3]{CP17}.
This tells us that, if $\psi$ is the Cauchy transform of $\overline{\phi}$ on $W(A)$, namely
\[
\psi(z):=\frac{1}{2\pi i}\int_{\partial W(A)} \frac {\overline{\phi}(\zeta)}{\zeta -z}\,d\zeta
\quad(z\in W(A)^\circ),
\]
then
\begin{equation}\label{E:CP}
\|\phi(A)+\psi(A)^*\|\le 2.
\end{equation}
Recalling that $\phi(A)$ attains its norm on the unit vector $x$, we deduce that
\begin{align*}
\|\phi(A)\|^2
&=\|\phi(A)x\|^2\\
&=\|(\phi(A)+\psi(A)^*)x\|^2-2\Re\langle \phi(A)x,\psi(A)^*x\rangle-\|\psi(A)^*x\|^2\\
&\le 4-2\Re\langle \psi(A)\phi(A)x,x\rangle.
\end{align*}
Thus, to complete the proof that $\|\phi(A)\|\le2$,
it suffices to show that
\begin{equation}\label{E:ineqtoprove}
\Re\langle \psi(A)\phi(A)x,x\rangle\ge0.
\end{equation}

To do this, we  compute $\phi(A)$ and $\psi(A)$.
Since $A$ has eigenvalues $\pm1$, it can be written as
\[
A=S^{-1}\begin{pmatrix}1 &0\\0 &-1\end{pmatrix}S,
\]
where $S$ is an invertible $2\times 2$ matrix. Hence
\begin{equation}\label{E:phi(A)}
\phi(A)=S^{-1}\begin{pmatrix}\phi(1) &0\\0 &\phi(-1)\end{pmatrix}S
=S^{-1}\begin{pmatrix}\phi(1) &0\\0 &-\phi(1)\end{pmatrix}S=\phi(1)A,
\end{equation}
where the second equality uses the fact that $\phi$ is odd.
Also $|\phi|=1$ on $\partial W(A)$, so $\overline{\phi}=1/\phi$ there, and so,
by the residue theorem,
\[
\psi(z)
=\frac{1}{2\pi i}\int_{\partial W(A)} \frac {1/\phi(\zeta)}{\zeta -z}\,d\zeta
=\frac{1}{\phi(z)}-\frac{1}{\phi'(0) z} \quad(z\in W(A)^\circ\setminus\{0\}).
\]
From \eqref{E:A} we have $\sigma(A)=\{-1,1\}\subset W(A)^\circ\setminus\{0\}$, and so
it follows that
\begin{equation}\label{E:psi(A)}
\psi(A)=\phi(A)^{-1}-\frac{1}{\phi'(0)}A^{-1}.
\end{equation}
Combining \eqref{E:phi(A)} and \eqref{E:psi(A)}, we obtain
\[
\psi(A)\phi(A)=\Bigl(1-\frac{\phi(1)}{\phi'(0)}\Bigr)I.
\]
Thus the desired inequality \eqref{E:ineqtoprove} will follow if we can show that
$\phi(1)\le\phi'(0)$.

To summarize, we have shown that the task of proving Theorem~\ref{T:Cr2x2}
reduces to showing that, if $\phi$ denotes the unique conformal mapping
of the ellipse $W(A)^\circ$ onto $\mathbb{D}$ satisfying $\phi(0)=0$ and $\phi'(0)>0$, then
\begin{equation}\label{E:ineqtoprove2}
\phi(1)\le\phi'(0).
\end{equation}
To this end, we invoke Theorem~\ref{T:bi-circsym}, 
applied to the map $\phi^{-1}:\mathbb{D}\to W(A)^\circ$.
By that theorem, $\phi^{-1}|_{[0,1)}$
is a positive, strictly increasing, convex function.
By Lemma~\ref{L:concaveinverse},
it follows that the restriction of $\phi$ to
$W(A)^\circ\cap\mathbb{R}^+$
is a positive, strictly increasing, concave function.
This clearly implies that $\phi(x)/x\le\phi'(0)$
for all $x\in W(A)\cap\mathbb{R}^+$.
In particular \eqref{E:ineqtoprove2} holds.
The proof of Theorem~\ref{T:Cr2x2} is complete.
\end{proof}

\begin{remarks}
(i) The idea of approaching the Crouzeix conjecture using the inequality \eqref{E:ineqtoprove}
was already remarked
by Schwenninger and de Vries in \cite[Theorem~6.1]{SdV25}.

\medskip

(ii) In \cite{CGL17}, the inequality \eqref{E:ineqtoprove2} is proved as follows. Let $\Omega$ be the ellipse $\{x^2/a^2+y^2/b^2<1\}$,
where $0<b<a$ and $a^2=b^2+1$. It can be shown that the conformal mapping $\phi$ of $\Omega$ onto $\mathbb{D}$
such that $\phi(0)=0$ and $\phi'(0)>0$ has the  form
\begin{equation}\label{E:phiform}
\phi(z)=\frac{2z}{\rho}\exp\Bigl(\sum_{n=1}^\infty \frac{2(-1)^nT_{2n}(z)}{n(1+\rho^{4n})}\Bigr),
\end{equation}
where $T_{2n}$ denotes the $2n$-th Chebyshev polynomial and $\rho:=a+b$ (see \cite{Sz50}).
Then
\begin{equation}\label{E:formulas}
\phi(1)=\frac{2}{\rho}\exp\Bigl(\sum_{n=1}^\infty \frac{2(-1)^n}{n(1+\rho^{4n})}\Bigr)
~~\text{and}~~
\phi'(0)=\frac{2}{\rho}\exp\Bigl(\sum_{n=1}^\infty \frac{2}{n(1+\rho^{4n})}\Bigr),
\end{equation}
from which the desired inequality \eqref{E:ineqtoprove2}  follows.

The proofs of Theorem~\ref{T:Cr2x2}  given in \cite{BCD06} and \cite{Cr04} also pass via the
representation \eqref{E:phiform},
but depend on the inequality
\begin{equation}\label{E:harderineq}
\phi(1)\le 2/\rho.
\end{equation}
Using the formulas \eqref{E:formulas}, we can see that
\[
2/\rho< \phi'(0).
\]
In fact, as remarked in \cite{Sz50},  this is an expression of the fact
that the inner radius of an ellipse is less than the outer radius.
Thus the inequality \eqref{E:harderineq}  is (at least formally) harder to prove than \eqref{E:ineqtoprove2}.

\end{remarks}


\section{Concluding remarks}
An eventual proof of the Crouzeix conjecture for $n\times n$ matrices would presumably include, as a special case, a proof for the $2\times 2$ case that does not depend on the special properties of ellipses.
This motivated us to seek such a proof of the $2\times 2$ case.

In the course of this project, we indeed found several such proofs, 
based on a number of different ideas, including symmetrization inequalities, conformal metrics
and the maximum principle.
Ultimately,  none of these appear in the paper, 
because we discovered a simpler approach  based on the work  of 
Jenkins \cite{Je55} on circularly symmetric functions.

We believe that the function-theory  results that we have developed along the way
are of interest in their own right, as well as possibly pointing the way
to an eventual proof of the Crouzeix conjecture.

  \section*{Acknowledgements}
 Part of this work was done while A.M.~and O.R.~ were visiting
 Universit\'e Laval, supported by a grant from the Simons Foundation. They  express their thanks for the wonderful
 atmosphere they experienced there.


\bibliographystyle{plain}
\bibliography{bibfile}

\begin{thebibliography}{10}

\bibitem{BCD06}
C.~Badea, M.~Crouzeix, and B.~Delyon.
\newblock Convex domains and {K}-spectral sets.
\newblock {\em Math. Z.}, 252(2):345--365, 2006.

\bibitem{CGL17}
T.~Caldwell, A.~Greenbaum, and K.~Li.
\newblock Some extensions of the {C}rouzeix--{P}alencia result.
\newblock preprint, arXiv:1707.08603v1, 2017.

\bibitem{Cr04}
M.~Crouzeix.
\newblock Bounds for analytical functions of matrices.
\newblock {\em Integral Equations Oper. Theory}, 48(4):461--477, 2004.

\bibitem{CGH14}
M.~Crouzeix, F.~Gilfeather, and J.~Holbrook.
\newblock Polynomial bounds for small matrices.
\newblock {\em Linear Multilinear Algebra}, 62(5):614--625, 2014.

\bibitem{CP17}
M.~Crouzeix and C.~Palencia.
\newblock The numerical range is a {{\((1+\sqrt{2})\)}}-spectral set.
\newblock {\em SIAM J. Matrix Anal. Appl.}, 38(2):649--655, 2017.

\bibitem{Du83}
P.~L. Duren.
\newblock {\em Univalent functions}, volume 259 of {\em Grundlehren der
  mathematischen Wissenschaften [Fundamental Principles of Mathematical
  Sciences]}.
\newblock Springer-Verlag, New York, 1983.

\bibitem{Fo19}
R.~Fournier.
\newblock On {J}ack's lemma.
\newblock {\em Rocky Mountain J. Math.}, 49(6):1869--1875, 2019.

\bibitem{Ja71}
I.~S. Jack.
\newblock Functions starlike and convex of order {$\alpha $}.
\newblock {\em J. London Math. Soc. (2)}, 3:469--474, 1971.

\bibitem{Je55}
J.~A. Jenkins.
\newblock On circularly symmetric functions.
\newblock {\em Proc. Amer. Math. Soc.}, 6:620--624, 1955.

\bibitem{KS06}
S.~Kanas and T.~Sugawa.
\newblock On conformal representations of the interior of an ellipse.
\newblock {\em Ann. Acad. Sci. Fenn. Math.}, 31(2):329--348, 2006.

\bibitem{SdV25}
F.~L. Schwenninger and J.~de~Vries.
\newblock The double-layer potential for spectral constants revisited.
\newblock {\em Integral Equations Operator Theory}, 97(2):Paper No. 13, 22,
  2025.

\bibitem{Sz50}
G.~Szeg\H{o}.
\newblock Conformal mapping of the interior of an ellipse onto a circle.
\newblock {\em Amer. Math. Monthly}, 57:474--478, 1950.

\end{thebibliography}

\end{document}